\documentclass{zariski-small}

\title{Projective Space in Synthetic Algebraic Geometry}

\begin{document}

\author{Felix Cherubini, Thierry Coquand, Matthias Ritter and David Wärn}

\maketitle

\begin{abstract}
  Synthetic algebraic geometry is a new approach to algebraic geometry. It consists in using homotopy type theory extended with three axioms, together with the interpretation of these in a higher version of the Zariski topos, in order to do algebraic geometry internally to this topos.
  
  In this article, we will show basic properties of projective n-space $\bP^n$ in synthetic algebraic geometry.
  In particular, we show that the automorphism group of $\bP^n$ is $\PGL_{n+1}(R)$ and that the picard group is $\Z$.
  We will provide different proofs of the latter statement, where the most synthetic approach naturally leads to the refined statement that the type of line bundles on $\bP^n$ is the higher type $\Z\times K(R^\times,1)$, where $K(R^\times,1)$ is a delooping of the group of units of the internal base ring $R$.
\end{abstract}

\section*{A brief introduction to synthetic algebraic geometry}

In mathematics, it is common practice to assume a fixed set theory, usually with the axiom of choice, as a common basis. While it is a great advantage to work in one common language and share a lot of the basic constructions, the dual approach of adapting the  ``base language'' to particular mathematical domains is sometimes more concise, provides a new perspective and encourages new proof techniques which would be hard to find otherwise.
We use the word ``synthetic'' to indicate that the latter approach is used,
as it was used by Kock and Lawvere to describe developement of mathematics internal to certain categories \cite{lawvere-categorical-dynamics}, in particular toposes -- a program which dates back as far as 1967.

Already in the 70s in the same program, Anders Kock suggested to use the language of higher-order logic \cite{Church40} to describe the Zariski topos, the collection of sheaves for the Zariski topology \cite{Kock74,kockreyes}, which is the first occurence of synthetic algebraic geometry.
Kock's approach allowed for a more suggestive and geometrical description of schemes.
There is in particular a ``generic local ring'' $R$, which, as a sheaf, associates to any algebra $A$ its underlying set and, as described in \cite{kockreyes}, the projective space $\bP^n$ is then the set of lines in $R^{n+1}$.

Just using category theory is not the same as reasoning synthetically -- for the latter the goal is usually to derive results exclusively in one system,
as Kock and Lawvere did with differential geometry in his work.
The distinction with just using an abstraction like categories is important, since the translation from the synthetic language and back can become cumbersome -- although it is still the goal to derive statements about ordinary mathematical objects in the end.

Starting with Kock and Lawvere's work, more differential geometry was developed synthetically \cite{kock-sdg} along with a study of the models of the theory \cite{moerdijk-reyes}.
One basic axiom of the theory, called the Kock-Lawvere axiom, allows for reasoning with nilpotent infinitesimals. Our version of synthetic algebraic geometry uses a generalisation of this axiom called the duality axiom. Let us now describe the Kock-Lawvere axiom.

The Kock-Lawvere axiom is added to a basic language which can be interpreted in good enough categories, for example toposes. More precisely, we need basic objects like $\emptyset$, $\{\ast\}$ and $\N$ as well as natural constructions like $A\times B$ or $A^B$ for objects $A$, $B$. These constructions come with data, like the projections in the case of $A\times B$, satisfying natural laws. We also need predicates $P(x)$ for elements $x:A$ so we can form subobjects like $\{x:A\mid P(x)\}$.
In this language, we assume there is a fixed ring $R$, which can be thought of as the real numbers. We define $\D(1)=\{x\in R\mid x^2=0\}$ to be the set of all square-zero elements of $R$, then the Kock-Lawvere axiom gives us a bijection
\[ e : R \times R  \to R^{\D(1)} \]
which commutes with evaluation at $0$ and projection to the first factor.
The intuition is that $\D(1)$ is so small that any function on it is linear and therefore determined by its value and its derivative at $0\in\D(1)$.
With this axiom, the derivative at $0:R$ of a function $f : R \to R$ may then be defined as $\pi_2(e^{-1}(f_{\vert \D(1)}))$. This is the start of a convenient development of differential calculus, which doesn't require any further structures on $R$ or other objects. This is the core of the synthetic method: we can work with these differential spaces as if they were sets.

To give an example, the tangent bundle of a manifold $M$ can be defined as $M^{\D(1)}$ and vector fields as sections of the map $M^{\D(1)}\to M$ evaluating at $0$. Then it is easy to see that a vector field is the same as a map $\zeta:\D(1)\to M^M$ with $\zeta(0)=\id_M$, which can be interpreted as an infinitesimal transformation of the identity map. This style of reasoning with spaces as if they were sets is also central in current synthetic algebraic geometry. 

The Kock-Lawvere axiom above are incompatible with the law of excluded middle (LEM) and therefore also with the axiom of choice (AC). Indeed they imply all maps from say $R$ to $R$ are differentiable, which contradicts LEM. 
However restricted versions of LEM and AC are compatible with this axiom. A very basic example is that equality of natural numbers is decidable, meaning that two natural numbers are either equal or not equal. We will latter go back to why full LEM and AC tend to be incompatible with synthetic approach to various geometry.


The use of nilpotent elements to capture infinitesimal quantities as mentioned above was inspired by the Grothendieck school of algebraic geometry and Anders Kock also worked with an extended axiom \cite{Kock74,kockreyes} suitable for synthetic algebraic geometry, where the role of $\D(1)$ above can be taken by any finitely presented affine scheme. In his 2017 doctoral thesis, Ingo Blechschmidt noticed a property holding internally in the Zariski-topos, which he called synthetic quasi-coherence. It is generalised and internalised version of what Kock used. In 2018, David Jaz Myers\footnote{Myers' never published on the subject, but communicated his ideas to Felix Cherubini and in talks to a larger audience \cite{myers-talk1,myers-talk2}.} started working with a specialization of Blechschmidt's synthetic quasi-coherence, which is what we now call \emph{duality axiom}.

To state the duality axiom we need to go from the space $\D(1)$ to spaces that are the common zeros of some finite system of polynomial equations over $R$. Such a space can be encoded independently of the choice of polynomials as a finitely presented $R$-algebra, i.e.\ an $R$-algebra $A$ which is of the form $R[X_1,\dots,X_n]/(P_1,\dots,P_l)$ for some numbers $n,l$ and polynomials $P_i\in R[X_1,\dots,X_n]$.
Then the set of roots of the system is given by the type $\Hom_{\Alg{R}}(A,R)$ of $R$-algebra homomorphisms from $A$ to the base ring. We denote this type by $\Spec A$.
Now the duality axiom states that $\Spec$ is the inverse to exponentiating with $R$, i.e.\ for all 
finitely presented $R$-algebras $A$ the following is an isomorphism:
\[ (a\mapsto (\varphi\mapsto \varphi(a))) : A\to R^{\Spec A}\rlap{.}\]

Myers used homotopy type theory as a base language, which is now the standard in synthetic algebraic geometry. Now we introduce homotopy type theory, in the next paragraphs we will explain how it fits with synthetic algebraic geometry. Homotopy type theory is a language for synthetic homotopy theory.
This means that when using it, we can think of the basic objects of the theory, that is types, directly as homotopy types. This should be contrasted with the usual practice in algebraic topology, which is to implement these homotopy types as topological spaces or Kan complexes.
So the rules of homotopy type theory allow to work with types in very much the same way as one would work with homotopy types in traditional mathematics. 

On the other hand we also use homotopy type theory because it allows to reason synthetically about spaces, as plain type theory does. A key point is that we do not use the law of excluded middle (LEM) or the axiom of choice (AC), which are incompatible with types being interpreted as spaces. Indeed on one hand LEM allows us to find a complement of each subset of a given type $A$, which exposes $A$ as a coproduct.
This is not true for spaces, for example, $\R$ is not the coproduct of the topological subspaces $\{0\}$ and $\R/\{0\}$.
On the other hand AC states that any surjection has a section. This is also not true for any sensible notion of space, in particular it would trivialise all cohomology.
Thus, constructive reasoning in the sense of not using LEM and AC is a necessity if we want to types to be understood as having a spatial structure. It turns out that this is the only obstruction, so the rules of type theory allow to work with type as one would work with spaces in algebraic geometry.

In synthetic algebraic geometry, we work inside homotopy type theory so that types behave both as homotopy types and as spaces from algebraic geometry. This means that we are mixing two synthetic approaches, combining their advantages,
which rests on the possibility of interpreting homotopy type theory in any higher topos \cite{shulman2019all} and not just the higher topos of $\infty$-groupoids. More precisely we think of the higher topos of Zariski sheaves with value in homotopy type. 
The general idea of using homotopy type theory to combine some kind of synthetic, spatial reasoning with synthetic homotopy theory, goes back at least to 2014, to Mike Shulman and Urs Schreiber \cite{Schreiber_2014}.
Schreiber suggested to the HoTT community at various occasions to make use of HoTT as the internal language of higher toposes, where specifities of the topos are accessed in the language via modalities.
This approach was shown to be quite effective and intuitive in Shulman's \cite{shulman-Brouwer-fixed-point} work on mixing synthetic homotopy theory in the form of HoTT and a synthetic approach to topology using a triple of modalities -- a structure called cohesion by Lawvere \cite{Lawvere2007}. A more detailed introduction to homotopy type theory for a general mathematical audience, with an emphasis on this mix of homotopical and spatial structure can be found in \cite{shulman-logic-of-spaces}.

One of the main advantages of using specifically homotopy type theory and not plain type theory, is using synthetic homotopical reasoning to make cohomological computations. Indeed one of Schreiber's motivation was to make use of the modern perspective on cohomology as the connected components of a space of maps in a higher topos. This can be mimicked in HoTT as follows: Given $X$ a type, $A$ an abelian group and $n:\N$, we define the $n$-th cohomology group of $X$ with coefficients in $A$ as
\[ H^n(X,A):=\| X\to K(A,n) \|_0\]
 where $\|\_\|_0$ is the $0$-truncation, an operation which turns any type into a $0$-type, that is a type with trivial higher structure. The type $K(A,n)$ is the $n$-th Eilenberg MacLane space, which can always be constructed for any abelian group $A$ and comes with an isomorphism $\Omega^n(K(A,n))\simeq A$.
With this definition of cohomology groups we can use synthetic homotopy theory to reason about cohomology, which had already been done successfully for the cohomology of homotopy types like spheres and finite cell complexes. It also works for the cohomology of $0$-types such as spaces in synthetic algebraic geometry.
This internal version of cohomology does not agree with the external version mentioned above, indeed it is a sheaf of groups instead of a single group, and it is indexed by an internal natural number instead of an external one. Nevertheless, internal cohomology turned out to be quite useful in practice.

In 2022, trying to use this approach to calculate cohomology groups in synthetic algebraic geometry led to the discovery of what is now called Zariski-local choice \cite{draft},
which is an additional axiom that holds in the higher Zariski-topos.
It is a weakening of the axiom of choice. In homotopy type theory, the axiom of choice can be formulated as follows: For any surjective map $f:X\to Y$, there exists a section, i.e.\ a map $s:Y\to X$ such that $f\circ s=\id_Y$.
Zariski-local choice also states the existence of a section, but only Zariski-locally and only for surjections into an affine scheme: For any surjection $f:E\to \Spec A$,
there exists a Zariski-cover $U_1,\dots,U_n$ of $\Spec A$ and maps $s_i:U_i\to E$ such that $f(s_i(x))=x$ for all $x\in U_i$.

In homotopy type theory, we use the propositional truncation $\|\_\|$ to define surjections and more generally what we mean by ``exists''.
Propositional truncation turns an arbitrary type $A$ into a type $\|A\|$ with the property that $x=y$ for all $x,y:\|A\|$.
Types with this property are called propositions or (-1)-types in homotopy type theory.
Using a univalent universe of types $\mathcal U$ we have that surjection into a type $A$ are the same as type families $F:A\to \mathcal U$, such that we have $\|F(x)\|$ for all $x: A$.
Using type families instead of maps allows us to drop the condition that the maps we get are sections, since we can express it using dependent function types and we arrive at the formulation of Zariski-local choice given below in the list of axioms.
In this instance and many others, homotopy type theory is much more convenient for formal reasoning, which is an advantage when formalizing synthetic algebraic geometry.

In total, the system we use for synthetic algebraic geometry consists of the extension of homotopy type theory postulating a fixed commutative ring $R$ satisfying these three axioms (see below for an explanation of the first one):

\begin{center}
\begin{axiom}[Locality]%
  \label{loc-axiom}
  $R$ is a local ring, i.e.\ $1\neq 0$ and whenever $x+y$ is invertible then $x$ is invertible or $y$ is invertible.
\end{axiom}

\begin{axiom}[Duality]%
  \label{duality-axiom}
  For any finitely presented $R$-algebra $A$, the homomorphism
  \[ a \mapsto (\varphi\mapsto \varphi(a)) : A \to (\Spec A \to R)\]
  is an isomorphism of $R$-algebras.
\end{axiom}

\begin{axiom}[Zariski-local choice]%
  \label{Z-choice-axiom}
  Let $A$ be a finitely presented $R$-algebra
  and let $B : \Spec A \to \mU$ be a family of inhabited types.
  Then there exists a Zariski-cover $U_1,\dots,U_n\subseteq \Spec A$
  together with dependent functions $s_i : (x : U_i)\to B(x)$.
\end{axiom}
\end{center}

As we explained above the duality axiom is a generalisation of the Kock-Lawvere axiom, which was used for convenient infinitesimal computations. It has a lot of consequences. In line with classical algebraic geometry, it shows that we have an anti-equivalence between finitely presented $R$-algebras and affine schemes of finite presentation over $R$.
More surprisingly, it implies that all functions in $R\to R$ are polynomials and that the base ring $R$ has surprisingly strong properties.
For example, for all $x:R$, we have that $x$ is invertible if and only if we have $x\neq 0$.

Surprisingly, the Zariski-local choice axiom was also usable to solve problems which have no obvious connection to cohomology.
For example, it implies that two reasonable definitions of open subsets agree.
In more detail, we can define open subsets using open propositions, which are propositions of the form $r_1\neq 0 \vee\dots\vee r_n\neq 0$ where $r_i:R$.
A subset $U$ of a type $X$ is open if the proposition $x\in U$ is open for all $x:X$.
Given an open subset $U$ of $\Spec A$, using Zariski-local choice we turn these elements $r_1,\dots,r_n$ of the base ring into functions defined Zariski-locally on $\Spec A$.
We can then even prove that $U$ is a union of non-vanishing sets $\mathrm{D}(f_i)$ of global functions $f_i:\Spec A \to R$, which is the second candidate for a definition of open subset alluded to above.
An analogous result holds for closed propositionsand vanishing sets of functions on affine schemes, where closed propositions are propositions of the form $r_1=0\wedge\dots\wedge r_n=0$ where $r_i:R$.

This connection between pointwise and Zariski-local openness is crucial to make the synthetic definition of a scheme work well:
A scheme is a type $X$, that merely has a finite open cover by affine schemes.
To produce interesting examples, it is necessary to use the locality axiom.
This is related to the Zariski topology and ensures that classical examples of Zariski covers can be reproduced.
A central example are the projective spaces $\bP^n$, which can be defined as the quotients of $R^{n+1}/\{0\}$ by the action of $R^\times$ by scaling.
A cover of $\bP^n$ is given by sets of equivalence classes of the form $\{[x_0:\cdots:x_n] \vert x_i\neq 0 \}$, which is clearly open using the pointwise definition.
To see that it is a cover, one has to note that for $x:R^{n+1}$, we have that $x\neq 0$ is equivalent to one of the entries $x_i$ being different from $0$. In synthetic algebraic geometry, this is the case for the base ring $R$ and the proof uses that $R$ is a local ring.

\section*{Content of this article}
\ignore{
Grothendieck advocated for a functor of points approach to schemes early on in his
project of foundation of algebraic geometry (see the introduction of \cite{EGAI}).
In this approach, a scheme is defined as a special kind of (covariant) set valued functor
on the category of commutative rings. This functor should in particular
be a sheaf w.r.t.\ the Zariski topology. As a typical example, the projective space $\bP^n$
is the functor, which to a ring $A$,
associates the set of sub-modules of  rank $1$ of $A^{n+1}$, which are direct factors \cite{Demazure,Eisenbud,Jantzen}.

In the 70s, Anders Kock suggested to use the language of higher-order logic \cite{Church40}
to describe the Zariski topos, the collection of sheaves for the Zariski topology \cite{Kock74,kockreyes}.
This allows for
a more suggestive and geometrical description of schemes.
There is in particular a ``generic
local ring'' $R$, which associates to $A$ its underlying set and, as described in \cite{kockreyes},
the projective space $\bP^n$ is then the set of lines in $R^{n+1}$.

In \cite{draft}, we presented an axiomatisation of the Zariski {\em higher topos} \cite{lurie-htt},
using instead of the language of higher-order logic the language of dependent type theory
with univalence \cite{hott}. The first axiom is that we have a local ring $R$. We then define
an affine scheme to be a type of the form $\Spec(A) = \Hom_{\Alg{R}}(A,R)$ for some finitely presented
$R$-algebra $A$. The second axiom, inspired from the work of Ingo Blechschmidt \cite{ingo-thesis},
states that the evaluation map $A\rightarrow R^{\Spec(A)}$ is a bijection. The last axiom states
that each $\Spec(A)$ satisfies some form of local choice \cite{draft}. We can then define a notion
of {\em open} proposition, with the corresponding notion of open subset, and define a scheme as a type
covered by a finite number of open subsets that are affine schemes. In particular, we define
$\bP^n$ as in \cite{kockreyes} and show that it is a scheme, by the usual covering by $n+1$
open affine subsets.
}

A natural question is if we can show in synthetic algebraic geometry that the automorphism group of $\bP^n$ is  $\PGL_{n+1}(R)$.
More generally, can we show that any map $\bP^n\rightarrow \bP^m$ is given by $m+1$ homogeneous
polynomials of same degree in $n+1$ variables?
From this, it is possible to deduce the corresponding result about $\bP^n$ defined as
a functor of points (but the maps are now {\em natural transformations}) or about $\bP^n$ defined
as a scheme (but the maps are now {\em maps of schemes}).
While this result is a basic fact in the case of projective space over a field, the general
case is more subtle.
(This general result, though fundamental, is not in \cite{Hartshorne} for instance.)
One goal of this paper is to present such a proof in the synthetic setting of \cite{draft}.

Interestingly, though these results can be stated
in the Zariski $1$-topos, the proof makes use of types that
are {\em not} (homotopy) sets (in the sense of \cite{hott}).
The argument proceeds as follows. The first step is to show that any line bundle on $\bP^1$
is trivial on each standard affine chart. The proof here is a reformulation of the standard
argument which follows from Horrocks Theorem \cite{Horrocks,Quillen,lombardi-quitte,Lam}.
One way to generalise this result to $\bP^n$ is to use the technique
of Quillen patching \cite{Quillen,lombardi-quitte,Lam}. 
We present an alternative argument which proceeds in characterizing $\bP^1\rightarrow\KR$
as $\ints\times\KR$, where $\KR$ is the delooping
(thus a type which is not a set) of the multiplicative group of units of $R$. Once we have
the result that a line bundle on $\bP^n$ is trivial on each standard affine chart, we
can deduce by a purely algebraic result (Proposition \ref{units}) that
$\Pic(\bP^n) = \ints$ and $\Aut(\bP^n) = \PGL_{n+1}(R)$ and that
any map $\bP^n\rightarrow \bP^m$ is given by $m+1$ homogeneous
polynomials of same degree in $n+1$ variables. We also provide an alternative purely geometric
proof that $\Pic(\bP^n) = \ints$ which uses our characterisation of $\bP^1\rightarrow \KR$.










\paragraph{Acknowledgements.}
Many thanks to Claude Quitt\'e, who suggested the use of Horrocks' Theorem, and found many
inacurracies in a previous version of this document.
Work on this article was supported by the ForCUTT project, ERC advanced grant 101053291.

\section[Definition of projective space and some linear algebra]{Definition of $\bP^n$ and some linear algebra}
We follow the notations and setting for Synthetic Algebraic Geometry \cite{draft}.
In particular, $R$ denotes the generic local ring and $R^\times$ is the multiplicative group of units of $R$.

In Synthetic Algebraic Geometry, a scheme is defined as a set satisfying some property \cite[Definiton 5.1.1]{draft}. In particular
the projective space $\bP^n$ can be defined to be the quotient of $R^{n+1}\setminus\{0\}$ by the
equivalence relation $a\sim b$ which expresses that $a$ and $b$ are proportional, i.e.\ $a_ib_j=a_jb_i$. This proportionality is a closed proposition which is equal to $\Sigma_{r:R^\times}ar = b$.
We can then prove \cite[Theorem 6.1.5]{draft} that this set is a scheme. This definition goes back to \cite{Kock74}.

 In this setting, a map of schemes is simply an arbitrary set theoretic map. An application of this work is to show
 that the maps $\bP^n\rightarrow \bP^m$ are given by $m+1$ homogeneous polynomials of the same degree in $n+1$ variables.

\medskip

There is another definition of $\bP^n$ which uses ``higher'' notions. Let $\KR= K(R^\times,1)$ be the delooping
of $R^\times$. It can be defined as the type of lines $\Sigma_{M:\Mod{R}}\|{M=R^1}\|$. Over $\KR$ we have the
family of sets
$$E_n(l) = l^{n+1}\setminus\{0\}$$
Note that we use the same notation for an element $l : \KR$,
its underlying $R$-module and its underlying set.
An equivalent definition of $\bP^n$ is then
$$
\bP^n = \sum_{l:\KR}E_n(l)
$$
This is a general construction of the \emph{homotopy} quotient of a type by a group action --
the type in this case is $E(R^1)=R^{n+1}\setminus\{0\}$ and the action of $R^\times=\GL(1,R)=\Omega (\KR,R^1)$ is given by transporting in $E(R^1)$ along a loop $l:\Omega (\KR,R^1)$.
One can calculate, that this transport and therefore the action corresponds to scalar multiplication.
Since this is a free group action, the homotopy quotient, given by the sigma type above, will be a set.
This is part of the homotopy typetheoretic view on group theory and group actions, which is explained in detail in \cite{Sym}.

We will use the following constructions of $\bP^n$ and the identifications between them given below:
\begin{remark}\label{identification-Pn}
  Projective $n$-space $\bP^n$ is given by the following equivalent constructions of which we prefer
  the first in this article:
  \begin{center}
  
    \begin{enumerate}[(i)]
    \item $\sum_{l:\KR}E_n(l)$
    \item The set-quotient $R^{n+1}\setminus\{0\}/R^\times$, where $R^\times$ acts on non-zero vectors in $R^{n+1}$ by multiplication.
    \item For any $k$ and $R$-module $V$ we define the \emph{Grassmannian}
      \[ \Gr(k,V)\colonequiv \{ U\subseteq V \mid \text{$U$ is an $R$-submodule and $\|U=R^k\|$ }\} \rlap{.}\]
      Projective $n$-space is then $\Gr(1,R^{n+1})$.
    \end{enumerate}
    
  \end{center}
  We use the following, well-defined identifications:
  \begin{center}
  
  \begin{enumerate}
  \item[] (i)$\to$(iii): Map $(l,s)$ to $R\cdot (u s_0,\dots, u s_n)$ where $u:l=R^1$
  \item[] (iii)$\to $(i): Map $L\subseteq R^{n+1}$ to $(L, x)$ for a non-zero $x\in L$
  \item[] (ii)$\leftrightarrow$ (iii): A line through a non-zero $x:R^{n+1}$
          is identified with $[x]:R^{n+1}\setminus\{0\}/R^\times$
  \end{enumerate}
    
  \end{center}
\end{remark}

We construct the standard line bundles $\OO(d)$ for all $d\in\Z$,
which are classically known as \emph{Serre's twisting sheaves} on $\bP^n$ as follows:

\begin{definition}
  For $d:\Z$, the line bundle $\OO(d):\bP^n\to \KR$ is given by $\OO(d)(l,s) = l^{\otimes d}$
  and the following definition of $l^{\otimes d}$ by cases:
  \begin{enumerate}[(i)]
  \item $d \geqslant 0$: $l^{\otimes d}$ using the tensor product of $R$-modules
  \item $d < 0$: $(l^{\vee})^{-d}$, where $l^{\vee}\colonequiv\Hom_{\Mod{R}}(l,R^1)$ is the dual of $l$.
  \end{enumerate}
\end{definition}

This definition of $\OO(d)$ agrees with \cite{draft}[Definition 6.3.2] where $\OO(-1)$
is given on $\Gr(1,R^{n+1})$ by mapping submodules of $R^{n+1}$ to $\KR$.
Using the identification of $\bP^n$ from \Cref{identification-Pn} we can give the following explicit equality:

\begin{remark}
  We have a commutative triangle:
  \begin{center}
    \begin{tikzcd}
        \sum_{l:\KR} E_n(l)\ar[rr]\ar[dr,swap,"\OO(1)"] && R^{n+1}\setminus\{0\}/R^\times\ar[ld,"\OO(1)"] \\
                  & \KR &
    \end{tikzcd}
  \end{center}

by the isomorphism given for $(l,s)$ by mapping $x:l$ to $r(u s_0,\dots, u s_n)\mapsto r(u x)$ for some isomorphism $u:l\cong R^1$.
\end{remark}

\medskip

 Connected to this definition of $\bP^n$, we will prove some equalities in the following.
 To prove these equalities, we will make use of the following lemma, which holds in synthetic algebraic geometry:
 
\begin{lemma}\label{invariant-implies-homogenous}
  Let $n,d:\N$ and $\alpha:R^n\to R$ be a map such that
  \[\alpha(\lambda x)=\lambda^d\alpha(x)\]
  then $\alpha$ is a homogenous polynomial of degree $d$.
\end{lemma}

\begin{proof}
  By duality, any map $\alpha:R^n\to R$ is a polynomial.
  To see it is homogenous of degree $d$, let us first note that any $P:R[\lambda]$ with $P(\lambda)=\lambda^d P(1)$
  for all $\lambda:R^\times$ also satisfies this equation for all $\lambda : R$ and is therefore homogenous of degree $d$.
  Then for $\alpha'_x:R[\lambda]$ given by $\alpha'_x(\lambda)\colonequiv \alpha(\lambda\cdot x)$
  we have $\alpha'_x(\lambda)=\lambda^d \alpha'_x(1)$. This means any coeffiecent of $\alpha'_x$
  of degree different from $d$ is 0. Since this means every monomial appearing in $\alpha$,
  which is not of degree $d$, is zero for all $x$ and therefore 0.   
\end{proof}

\begin{proposition}\label{end}
  $$\prod_{l:\KR}l^n\rightarrow l \;\;\;=\;\;\; \Hom(R^n,R)$$
\end{proposition}

\begin{proof}
We rewrite $\Hom(R^n,R)$, the set of $R$-module morphism, as
$$
\sum_{\alpha:R^n\rightarrow R}\prod_{\lambda:R^\times}\prod_{x:R^n}\alpha(\lambda x) = \lambda \alpha(x)
$$
using \Cref{invariant-implies-homogenous} with $d=1$.

\medskip

It is then a general fact that if we have a pointed connected groupoid $(A,a)$ and a family of
sets $T(x)$ for $x:A$, then $\prod_{x:A}T(x)$ is the set of fixed points of $T(a)$ for the $\Omega(A,a)$-action: If we have $f:\prod_{x:A}T(x)$, then for any $g:\Omega(A,a)$ we have $g(f_a)=f_a$ by depedent application of $f$ to $g$. And if $z:T(a)$ is a fixed point and $x:A$, the transport along a $p:a=x$ is independent of $p$, so we can construct a $z':T(x)$.
\end{proof}

We will use the following remark, proved in \cite{draft}[Remark 6.2.5].

\begin{lemma}\label{ext}
  Any map $R^{n+1}\setminus\{0\}\rightarrow R$ can be uniquely extended to a map $R^{n+1}\rightarrow R$ for $n>0$.
\end{lemma}

We will also use the following proposition, already noticed in \cite{draft}.

\begin{proposition}\label{const}
  Any map from $\bP^n$ to $R$ is constant.
\end{proposition}

\begin{proof}
  Since $\bP^n$ is a quotient of $R^{n+1}\setminus\{0\}$, the set $\bP^n\rightarrow R$ is
  the set of maps $\alpha:R^{n+1}\setminus\{0\}\rightarrow R$
  such that $\alpha(\lambda x) = \alpha(x)$ for all $\lambda$ in $R^\times$.
  These are exactly the constant maps
  using \Cref{ext} and \Cref{invariant-implies-homogenous} with $d=0$.
\end{proof}

\begin{proposition}\label{aut}
  For all $n:\N$ we have:
$$\prod_{l:\KR}E_n(l)\rightarrow E_n(l) \;\;=\;\; \GL_{n+1}$$
\end{proposition}

\begin{proof}
  For $n=0$, this is the direct computation that a Laurent-polynomial $\alpha:(R[X,1/X])^\times$ which satisfies
  $\alpha(\lambda x)=\lambda \alpha(x)$ is $\lambda\alpha(1)$ where $\alpha(1):R^\times=\GL_1$.
  
  \medskip
  
  For $n>0$, the proposition follows from two remarks.

  The first remark is that maps $E_n(R)\to E_n(R)$, which are invariant under the induced $\KR$ action, are linear.
  To prove this remark, we first map from $E_n(l)\to E_n(l)$ to $E_n(l)\to l^{n+1}$ by composing with the inclusion.
  Maps of the latter kind can be uniquely extended to maps $l^{n+1}\to l^{n+1}$, since by 
  \Cref{ext} the restriction map
$$
(l^{n+1}\rightarrow l)\rightarrow ((l^{n+1}\setminus\{0\})\rightarrow l)
$$
is a bijection for $n>0$ and all $l:\KR$.

\medskip

The second remark is that a linear map $u:R^{m}\rightarrow R^{m}$ such that
$$
x\neq 0~\rightarrow~u(x)\neq 0
$$
is exactly an element of $\GL_{m}$.

We show this by induction on $m$. For $m=1$ we have $u(1)\neq 0$ iff $u(1)$ invertible.

For $m>1$, we look at $u(e_1) = \Sigma \alpha_ie_i$ with $e_1,\dots,e_m$ basis of $R^m$.
We have that some $\alpha_j$ is invertible.
By composing $u$ with an element in $\GL_m$, we can then
assume that $u(e_1) = e_1+v_1$ and $u(e_i) = v_i$, for $i>1$, with $v_1,\dots,v_m$ in $Re_2+\dots+Re_m$.
We can then conclude by induction.
\end{proof}

We can generalize \Cref{end}
and get a result related to \Cref{aut} as follows.
 
\begin{lemma}\label{hom}
  \begin{enumerate}[(i)]
    \item
      \[  \prod_{l:\KR}l^n\rightarrow l^{\otimes d} \;\;=\;\; (R[X_1, \dots, X_n])_d \]
      That is,
      every element of the left-hand side is given by
      a unique homogeneous polynomial of degree $d$ in $n$ variables.
    \item
      An element in
      $$\prod_{l:\KR}E_n(l)\rightarrow E_m(l^{\otimes d})$$
      is given by $m+1$ homogeneous polynomials $p = (p_0,\dots,p_m)$ of degree $d$ such that
      $x\neq 0$ implies $p(x)\neq 0$.
  \end{enumerate}
\end{lemma}

\begin{proof}
We show the first item. Following \cite{Sym} again, this product is the set of maps $\alpha:R^n\rightarrow R^{\otimes d}$
which are invariant by the $R^\times$-action which in this case acts by mapping $\alpha$ to $r^d\alpha(r^{-1} x)$ for each $r:R^\times$.
So by \Cref{invariant-implies-homogenous} these are exactly the maps given by homogeneous polynomials of degree $d$.
\end{proof}


\section{Line bundles on affine schemes}
A line bundle on a type $X$ is a map $X\rightarrow \KR$.


\medskip

 A line bundle $L$ on $\Spec(A)$ will define a f.p. $A$-module $\prod_{x:\Spec(A)}L(x)$ \cite{draft}.
It is presented by a matrix $P$.
Since this f.p. module is locally free, we can find $Q$ such that $PQP = P$ and
$QPQ = Q$ \cite{lombardi-quitte}. We then have $Im(P) = Im(PQ)$ and this is a projective module of rank $1$. We can then assume $P$ square matrix and
$P^2 = P$ and the matrix $I-P$ can  be seen as listing the generators of this module.

If $M$ is a matrix we write $\Delta_l(M)$ for the ideal generated by the $l\times l$ minors of
$M$. We have $\Delta_1(I-P) = 1$ and $\Delta_2(I-P) = 0$, since this projective module is of rank $1$.

The module is free exactly if we can find a column vector $x$ and a line vector $y$ such that
$xy = I-P$. We then have $yx = 1$, since if $r = yx$ we have $I-P = xyxy = rxy = r(I-P)$ and
hence $r = 1$ since $\Delta_1(I-P) = 1$.

\medskip

The line bundle on $\Spec(A)$ is trivial on $D(f)$ if, and only if, the module $M\otimes A[1/f][X]$ is free, which
is equivalent to the fact that we can find $X$ and $Y$ such that $YX = (f^N)$ and $XY = f^N(I-P)$ for some $N$.

In Appendix 1, we prove the following special case of Horrocks' Theorem.

\begin{lemma}
  If $A$ is a commutative ring
  then any ideal of $A[X]$ divides a principal ideal $(f)$, with $f$ monic, is itself a principal ideal.
\end{lemma}

We can then apply this result in Synthetic Algebraic Geometry for the ring $R$.

\begin{proposition}
  If we have $L:\A^1\rightarrow \KR$ which is trivial on some $D(f)$ where $f$ in $R[X]$ is monic
  then $L$ is trivial on $\A^1$.
\end{proposition}

\begin{corollary}\label{c1}
  If we have $L:\bP^1\rightarrow \KR$ then we have
  $$\|{\prod_{r:R}L([1:r]) = L([1:0])}\|\,\,\,\,\,\,\,\,\,\,\,\,\,\|{\prod_{r:R}L([r:1]) = L([0:1])}\|$$
\end{corollary}

\begin{proof}
  We have the two open injections $i_0:\A^1\rightarrow \bP^1,~r\mapsto (r:1)$ and
  $i_1:\A^1\rightarrow \bP^1,~r\mapsto (1:r)$.
  By Zariski local choice \cite{draft}, the line bundle $L\circ i_0$ is locally trivial on $\A^1 = \Spec(R[X])$.
  In particular, there exists $g = a_0 + a_1 X + \dots + a_nX^n$ in $R[X]$ such that $0$ is in $D(g)$, i.e. $a_0\neq 0$
  and $L\circ i_0$ is trivial on $D(g)$. Let $f$ be $X(X^n + a_1X^{n-1}+\dots + a_n)$. We have $i_1(D(f))\subseteq D(g)$
  and hence $L\circ i_1$ is trivial on $D(f)$. Since $f$ is monic, $L\circ i_1$ is trivial by the previous Proposition.
  Similarly $L\circ i_0$ is trivial.
\end{proof}

\section{Application of the Veronese embedding}
For given $d>0$, we introduce the Veronese map $V: \bP^n\rightarrow \bP^N$ with $N = \binom{n+d}{n}-1$.
We write an element of $\bP^N$ as a sequence of elements $z_{i_0,\dots,i_n}$ in $R$, not all $0$,
indexed by $i_0,\dots,i_n$ such that $d = i_0+\dots+i_n$
It is defined by $V(x0:\dots:x_n)$ to be the elements $z_{i_0,\dots,i_n} = x_0^{i_0}\dots x_n^{i_n}$. It is well-defined
since we have $x_i^d\neq 0$ for some $i$.

\begin{proposition}\label{veronese}
  $V$ is a bijection between $\bP^n$ and the closed subset $V(n,d)$ of $\bP^N$ determined by the quadratic
  equations $z_Iz_J = z_Kz_L$ for $I+J=K+L$.
\end{proposition}

\begin{proof}
  We have $\bP^N$ covered by the open $U_0,\dots,U_n$ with $U_l$ set of $z_I$ with $z_I\neq 0$
  for some $I = i_0,\dots,i_n$ with $i_l>0$. On $U_l$ we define the map $g_l(z_I) = (x_0:\dots:x_n)$
  with $x_l = z_I$ and $x_k = z_{J_k}$ with $J_k = j_0,\dots,j_n$ and $j_l = i_l-1$ and $j_k = i_k+1$
  and $j_p = i_p$ for $p\neq l,k$. Using the quadratic relations, we see that the map $g_l$
  are compatible and define an inverse of the Veronese map.
\end{proof}

\begin{corollary}\label{affine}
  Let $q(x_0,\dots,x_n)$ be a homogeneous polynomial of degree $d>0$ (the polynomial $q$ might be $0$). Then $S_q$ the subset
  of $\bP^n$ of $(x_0:\dots:x_n)$ such that $q(x_0,\dots,x_n)\neq 0$
  defines an {\em affine} subset of $\bP^n$.
\end{corollary}

\begin{proof}
  To simplify the notation, we present the argument in the case $n=1$ and $d=2$, but the same can be done
  in general. We have $q = ax_0^2 + bx_0x_1+cx_1^2$. The Veronese embedding is $(x_0:x_1)\mapsto (x_0^2:x_0x_1:x_2^2)$
  and the subset $q\neq 0$ of $\bP^1$ is in bijection with the subset $ay_0+by_1+cy_2\neq 0, y_0y_2 = y_1^2$ of $\bP^2$.
  This subset is itself in bijection with the {\em affine} subset $ay_0+by_1+cy_2 = 1, ~y_0y_2=y_1^2$ of $R^3$.
\end{proof}

Note that we have the following description of $R^{S_q}$.

 \begin{lemma}
   $R^{S_q}$ is  $R[X_0,\dots,X_n][1/q]_0$.
\end{lemma}

 \begin{proof}
   Let $T_q$ be the set of $x$ in $R^{n+1}$ such that $q(x) \neq 0$. We can see $S_q$ as the quotient of $T_q$
   by the relation $x = sy$ for some unit $s$. So $R^{S_q}$ is the subset of elements $u$ in $R^{T_q} = R[X_0,\dots,X_n][1/q] = A$
   such that $u(sx) = u(x)$ for $s$ unit. Let $d$ be the formal degree of $q$, and let us write $u$ as $\Sigma a_i/q^l$ with
   $a_i$ homogeneous component of degree $i$. We have $s\neq 0\rightarrow \Sigma_{i\neq dl} a_is^i/s^{dl}q^l = 0$ and hence
   $a_i = 0$ in $A$ if $i\neq dl$.
 \end{proof}
 
 \begin{corollary}
   $R[X_0,\dots,X_n][1/q]_0$ is a finitely presented $R$-algebra.
 \end{corollary}

\section[Picard group of projective space]{Picard group of $\bP^1$}





 \begin{lemma}\label{stand}
   Let $A$ be a \emph{connected}\footnote{If $e(1-e) = 0$ then $e=0$ or $e=1$.} ring, then
   an invertible element of $A[X,1/X]$ can be written $X^N\Sigma a_nX^n$ with $N$ in $\Z$
   and $a_0$ unit and $a_n$ nilpotent if $n\neq 0$.
 \end{lemma}

 \begin{proof}
   Let $P = \sum_i a_iX^i:A[X,1/X]$ be invertible.
   The result is clear if $A$ is an integral domain. Let $B(A)$ is the constructible spectrum of $A$
   with the two generating maps $D(a)$ and $V(a)$ for $a$ in $A$ \cite{lombardi-quitte}. The argument for an integral
   domain, looking at $D(a)$ as $a\neq 0$ and $V(a)$ as $a =0$, shows that we have $\vee_i D(a_i) = 1$
   and $D(a_ia_j) = 0$ for $i\neq j$. Since $A$ is connected, this implies that exactly one $a_i$ is a unit,
   and all the other coefficient are nilpotent.
 \end{proof}

 We can prove in the same way.

 \begin{lemma}\label{nilpunit}
   For any commutative ring $A$, a polynomial $a_0+a_1X+\dots+a_nX^n$ is a unit in $A[X]$ iff
   $a_0$ is a unit and $a_1,\dots,a_n$ are nilpotent. A polynomial $a_0+a_1X+\dots+a_nX^n$ is nilpotent in $A[X]$ iff
   all $a_i$ are nilpotent.
 \end{lemma}
 
 \begin{lemma}\label{connected}
   If $A$ is connected, so is $A[X]$.
 \end{lemma}

 \begin{proof}
   Assume $A$ connected.
   Let $e(X) = a_0+a_1X+\dots+a_nX^n$ be an idempotent element in $A[X]$. Then $a_0$ is idempotent in $A$
   and hence $a_0 = 0$ or $a_0 = 1$. If $a_0 = 0$ then we have $a_1 = 0$ and then $a_2 = \dots = a_n = 0$
   and $e(X) = 0$. If $a_0 = 1$ then $1-e(X)$ is idempotent and we deduce similarly that $1 = e(X)$.
  \end{proof}

 Using Lemma \ref{stand}, we deduce the following, for $A$ connected ring.

\begin{lemma}\label{nilpotent}
  Any invertible element of $A[X,1/X]$ can be written uniquely as a product
  $uX^l(1+a)(1+b)$ with $l$ in $\Z$, $u$ in $A^{\times}$ and $a$ (resp. $b$)
  polynomial in $XA[X]$ (resp. $1/XA[1/X]$) with only nilpotent coefficients.
\end{lemma}

\begin{proof}
  Write $\Sigma v_nX^n$ the invertible element of $A[X,1/X]$.
  W.l.o.g. we can assume, by the previous Lemma, that the polynomial is of the form $v_0 + \Sigma v_nX^n$ with
  all $v_n,~n\neq 0$ nilpotent.
  We let $J$ be the ideal generated by these nilpotent elements.
  We have some $N$ such that $J^N = 0$.
  
  We first multiply by the inverse of $v_0 + \Sigma_{n<0}v_nX^n$, making all coefficients of
  $X^n,~n<0$ in $J^2$.
  We keep doing this until all these elements are $0$.
  We have then written the invertible polynomials on the form $u(1+a)(1+b)$ with $u$ unit in $A$.

  Such a decomposition is unique: if we have $(1+a)(1+b)$ in $A^{\times}$ with $a = \Sigma_{n> 0}a_nX^n$
  and $b = \Sigma_{n<0}b_nX^n$ then we have $a = b = 0$. 
\end{proof}

\begin{corollary}\label{Pic1}
  We have $\prod_{L:\bP^1\rightarrow \KR}\Sigma_{p:\Z}\|L = \OO(p)\|$
\end{corollary}

\begin{proof}
A line bundle $L([x_0,x_1])$ on $\bP^1$ is trivial on each of the affine charts $x_0\neq 0$ and $x_1\neq 0$ by Corollary \ref{c1}, so
it is characterised by an invertible Laurent polynomial on $R$, and the result follows from Lemma \ref{nilpotent}.
\end{proof}

We can then state the following strengthening.

\begin{proposition}\label{Matthias}
  The map $\KR\times\Z\rightarrow (\bP^1\rightarrow \KR)$
  which associates to $(l_0,d)$ the map $x\mapsto l_0\otimes \OO(d)(x)$ is an equivalence.
\end{proposition}

\begin{proof}
  Using Theorem 4.6.3 of \cite{hott}, it is equivalent to show that this map is both surjective and an embedding.
  Corollary \ref{Pic1} shows that this map is surjective.
  So we can conclude by showing that the map is also an embedding.
  For $(l,d),(l',d'):\KR\times\Z$ let us first consider the case $d=d'$. 
  Then we merely have $(l,d)=(\ast,d)$ and $(l',d')=(\ast,d)$,
  so it is enough to note that the induced map on loop spaces based at $(\ast,d)$ is an equivalence by \Cref{const}.
  Now let $d\neq d'$. To conclude we have to show $\OO(k)$ is different from $\OO(0)$ for $k\neq 0$.
  It is enough to show that for $k>0$ the bundle $\OO(k)$ has at least two linear independent sections,
  since we know $\OO(0)$ only has constant sections by \Cref{const}.
  This follows from the fact that $\OO(k)(x)$ is $\Hom_{\Mod{R}}(Rx^{\otimes k},R)$ and has all projections as sections.
\end{proof}

Note that the map $\Z\rightarrow (\bP^1\rightarrow \KR)$ which
associates to $d$ the map $x\mapsto \OO(d)(x)$ is an also surjective, but it is not an embedding.

 It is a curious remark that $\KR\rightarrow \KR$ is also equivalent
 to $\KR\times \Hom_{\mathrm{Group}}(R^\times,R^\times) = \KR\times\Z$.

\begin{corollary}\label{Matthias1}
  We have $\prod_{L:\bP^1\rightarrow \KR}\prod_{x:R}L([1:x]) = L([0:1])$.
\end{corollary}

\begin{proof}
  By the equivalence in \Cref{Matthias}, we have
  \[ \prod_{L:\bP^1\to \KR} \,\prod_{x : \bP^1}  L(x)=l_0\otimes \OO(d)(x) \]
  for some $(l_0,d)$ corresponding to $L$.
  $\OO(d)([0:1])$ can be identified with $R^1$ and $\OO(d)$ is trivial on $R$,
  so we have $L([1:x])=l_0=L([0:1])$ for all $x:R$.
\end{proof}

\section{Line bundles on $\bP^n$}
We will prove $\Pic(\bP^n)=\Z$ and a strengthening thereof in this section by mostly algebraic means.
In \Cref{geometric-proof} we will give a shorter geometric proof.

We can now reformulate Quillen's argument for Theorem 2' \cite{Quillen} in our setting.

\begin{proposition}\label{trivial}
  For all $V:\bP^n\rightarrow \KR$ we have ${\prod_{s:R^n}V([1:s]) = V([0:1:0:\cdots :0])}$.
\end{proposition}

\begin{proof}
  We define $L:R^{n-1}\rightarrow (\bP^1 \to \KR)$ by $L~t~[x_0:x_1] = V([x_0:x_1:x_0t])$.
  Let $s=(s_1,\dots,s_{n}):R^{n}$. We apply Corollary \ref{Matthias1} and we get
  \[
   V([1:s]) = L~(s_2,\dots,s_n)~[1:s_1] = L~(s_2,\dots,s_n)~[0:1] = V([0:1:0:\cdots :0])
   \rlap{.}
  \]
\end{proof}

 Note that the use of Corollary \ref{Matthias1} replaces the use of the ``Quillen patching''
 \cite{lombardi-quitte} introduced in \cite{Quillen} (see Appendix 3).

\medskip

Let $A$ be a commutative ring.
Let $T(A)$ be the ring of polynomials $u = \Sigma_p u(p)X^p$ with
$X^p = X_0^{p_0}\dots X_n^{p_n}$ with $\Sigma p_i = 0$. We write $T_l(A)$ for the subring
of $T(A)$ which contains only monomials $X^p$ with $p_i\geqslant 0$ if $i\neq l$
and $T_{lm}(A)$ the subring of $T(A)$ 
which contains only monomials $X^p$ with $p_i\geqslant 0$ if $i\neq l$ and $i\neq m$.

Note that $T_l(A)$ is the polynomial ring $T_l(A) = A[X_0/X_l,\dots,X_n/X_l]$.

A line bundle on $\bP^n$ is given by compatible line bundles on each $\Spec(T_l(R))$.

The result $\Pic(\bP^n) = \Z$ will be a consequence of the following result, proved in the Appendix.

\begin{proposition}\label{units}
  Let $A$ be a {\em connected} ring and let $t_{ij}$ be a family of  invertible elements in $T_{ij}(A)$ such that $t_{ii} = 1$
  and $t_{ij}t_{jk} = t_{ik}$. There exists an integer $N$ and $s_i$ invertible in $T_i(A)$ such that $t_{ij} = (X_j/X_i)^Ns_j/s_i$ 
\end{proposition}


\begin{corollary}
  $\Pic(\bP^n) = \Z$.
\end{corollary}

We can then strengthen this result, with the same reasoning as in Proposition \ref{Matthias}.

\begin{theorem}\label{Matthias2}
  The map $\KR\times\Z\rightarrow (\bP^n\rightarrow \KR)$
  which associates to $l_0,d$ the map $x\mapsto l_0\otimes \OO(d)(x)$ is an equivalence.
\end{theorem}

We deduce from this a characterisation of the maps $\bP^n\rightarrow\bP^m$.

\begin{corollary}\label{map}
  A map $\bP^n\rightarrow\bP^m$ is given by $m+1$ homogeneous polynomials $p = (p_0,\dots,p_m)$ on $R^{n+1}$
  of the same   degree $d$ such that $x\neq 0$ implies $p(x)\neq 0$.
\end{corollary}

\begin{proof}
Write $E_n(l)$ for $l^{n+1}\setminus\{0\}$. We have $\bP^n = \Sigma_{l:\KR}E_n(l)$ and so
$$
\bP^n\rightarrow\bP^m = \sum_{s:\bP^n\rightarrow \KR}\prod_{x:\bP^n}E_m(s~x)
$$
Using Theorem \ref{Matthias2}, this is equal to
$$
\sum_{l_0:\KR}\sum_{d:\Z}\prod_{l:\KR}E_n(l)\rightarrow E_m(l_0\otimes l^{\otimes d})
$$
and, as for Lemma \ref{hom}, this is the set of tuples of $m+1$ polynomials in $R[X_0,\dots,X_n]$ homogenenous
of degree $d$, sending $x\neq 0$ to $p(x)\neq 0$, and quotiented by proportionality.
\end{proof}

We deduce the characterisation of $\Aut(\bP^n)$. This is a
remarkable result, since the automorphisms are in this framework only bijections of sets.

\begin{corollary}
  $\Aut(\bP^n)$ is $\PGL_{n+1}(R)$.
\end{corollary}

It follows from Corollary \ref{affine} that $\PGL_{n+1}(R)$ is affine.

\medskip

Note that this is an {\em internal} result about functor of points over a category of
finitely presented $k$-algebras for some commutative ring $k$ \cite{draft}. We have an exact sequence
$$
1\rightarrow R^{\times}\rightarrow \GL_{n+1}(R)\rightarrow \PGL_{n+1}(R)\rightarrow 1
$$
and the global sections of $\PGL_{n+1}(R)$ is only $\PGL_{n+1}(k)$ if $\Pic(k)\neq 0$.

 In general, over a commutative ring,
there can be automorphisms of $\bP^n(k)$ which cannot be described as an element of $\GL_{n+1}(k)$.
For an example, let $k$ be the ring $\ints[x]$ with $x^2+5=0$ and let $m$ be the matrix
\[
m =
\begin{pmatrix}
  1+x & -2 \\
  -2 & 1-x \\
\end{pmatrix}  
\]
Its determinant is $2$, not unit in $k$, and $x\mapsto m^{-1}xm$ defines an automorphism of $M_2(k)$
which is {\em not} inner, but only locally inner \cite{Isaacs}, Example 6. This automorphism defines an
automorphism of $\bP^1(k)$, since a point of $\bP^1(k)$ is a submodule of $k^2$, projective of rank $1$ and factor direct,
which can be seen as an idempotent of $M_2(k)$. One can check that
this automorphim sends the free submodule $k(1,0)$ to the submodule generated by $(3,1+x)$ and $(-1+x,-2)$,
which is not free. This shows that this automorphism cannot be defined by an element of $\GL_2(k)$.

\medskip

We also have the following application of computation of cohomology groups \cite{draft}.

\begin{corollary}
A function $\bP^n\rightarrow\bP^m$ is constant if $n>m$.
\end{corollary}

\begin{proof}
We proved in \cite{cech-draft} that cohomology groups can be computed as Cech cohomology for any
finite open acyclic covering and used this to prove $H^n(\bP^n,\OO(-n-1))=R$.
By \Cref{map}, a map $\bP^n\rightarrow\bP^m$ is given by $m+1$ non zero polynomials
$p(x) = (p_0(x),\dots,p_m(x))$ homogeneous of the same degree $d\geqslant 0$ and such that $x\neq 0$ implies $p(x)\neq 0$.
This means that $\bP^n$ is covered by $m+1$ open subsets $U_i(x)$ defined by $p_i(x)\neq 0$.
I claim that we should have $d=0$.

 If $q(x)$ is a non zero homogeneous polynomial of degree $d>0$, the open $q(x)\neq 0$ defines an {\em affine}
 and hence acyclic \cite{cech-draft}, open subset of $\bP^n$ by Corollary \ref{affine}, and so is each finite
 intersection $U_{i_1}\cap\dots\cap U_{i_l}$. So if $d>0$, the covering $U_0,\dots,U_m$ is acyclic.
 If we compute $H^n(\bP^n,\OO(-n-1))$ with respect to {\em this} covering, we get that
 $H^n(\bP^n,\OO(-n-1)) = 0$ since $m<n$. But this contradicts $H^n(\bP^n,\OO(-n-1))=R$.

 Hence $d=0$ and the map is constant.
\end{proof}

\section[Automorphism group of projective space]{Another proof of $\Aut(\bP^n) = \PGL_{n+1}$}
In this section, we present an alternative way to prove that   $\Aut(\bP^n)$ is $\PGL_{n+1}$.

A projective module of rank $1$ over a polynomial
ring $K[X_1,\dots,X_n]$ is free, where $K$ is a discrete field, since this polynomial ring is a gcd domain, see
e.g. \cite{seminormal}. (The result actually holds in general
for rank $n$, this is the famous Serre's Problem \cite{Lam}.) Interpreting this proof dynamically
\cite{lombardi-quitte}, it follows that if $A$ is now an arbitrary commutative ring, and $M$ a
projective module of rank $1$ over $A[X_1,\dots,X_n]$, it is possible to build a binary tree, where
the root is $A$, and each node is an ring extension $i:A\rightarrow B$, and a branch is obtained
by forcing an element $a$ in $A$ to be invertible $B\rightarrow B[1/i(a)]$ or zero $B\rightarrow B/(i(a))$,
and each leaf is such that $M\otimes B[X_1,\dots,X_n]$ is free.

For the ring $R$, we have $a\neq 0\rightarrow \exists_x ax=1$ and hence $\neg \neg (a=0\vee \exists_x a x=1)$.
Hence we have the following result, by a descending induction over finite binary trees.

\begin{lemma}\label{notnot}
  A projective module of rank $1$ over $R[X_1,\dots,X_n]$ is not not free. Equivalently, 
a line bundle on $\Spec(R[X_1,\dots,X_n]) = R^n$ is not not trivial.
\end{lemma}

This generalizes Corollary 3.3.6 of \cite{draft}.


Similarly, using the fact that maps $\bP^n\rightarrow \bP^m$ are given by $m+1$ homogeneous polynomials in $K[X_0,\dots,X_n]$
over a discrete field $K$, we deduce.

\begin{lemma}\label{weakmap}
  Given $\varphi:\bP^n\rightarrow  \bP^m$, we have not not there exists
  $m+1$ homogeneous polynomials $p = (p_0,\dots,p_m)$ on $R^{n+1}$
  of the same   degree $d$ such that $x\neq 0$ implies $p(x)\neq 0$.
\end{lemma}




 We can then give a proof that   $\Aut(\bP^n)$ is $\PGL_{n+1}$ without relying on Horrocks' Theorem.

We see the element of $\bP^n$ as lines in $R^{n+1}$.
If $p_0,\dots,p_n$ are $n+1$ points of $\bP^n$, we can define the open proposition that
the points $p_0,\dots,p_n$ are independent
by the fact that for any choice of non zero elements $v_0,\dots,v_n$ in respectively $p_0,\dots,p_n$
the determinant of the $(n+1)\times (n+1)$ matrix $v_0\dots v_n$ is $\neq 0$.

\begin{corollary}
    $\Aut(\bP^n)$ is $\PGL_{n+1}$.
\end{corollary}

\begin{proof}
  Let $\varphi$ be an element of $\Aut(\bP^n)$.
  Let $u_0,\dots,u_n$ be the vectors $(1,0,\dots,0), \dots, (0,\dots,0,1)$
  and $u$ be the vector $(1,1,\dots,1)$.

  We can find a $(n+1)\times (n+1)$ matrix $A$ such that $\varphi(R u_i) = R (Au_i)$
  and $\varphi(R u) = R (A u)$. By Lemma \ref{weakmap}, we have not not $\varphi$ is represented by $A$,
  that is $\neg\neg \forall_{v\neq 0}\varphi(R v) = R (Av)$. If follows that $\neg \neg det(A) \neq 0$
  and hence $\det(A)\neq 0$ and $A$ is in $GL_{n+1}$.

  In this way, we reduce the problem of showing that if $\neg\neg \forall_{v\neq 0}\varphi (R v) = R v$
  then $\varphi$ is the identity. But then we have not not $\varphi(U_i)\subseteq U_i$
  and so $\varphi(U_i)\subseteq U_i$ for the standard $n+1$ affine charts $U_i$ of $\bP^n$, and we can
  conclude as in \cite{Demazure}, III, 4.6.
\end{proof}

\section[Picard group of projective space (geometric)]{A geometric proof of $\Pic(\bP^n)=\Z$}
\label{geometric-proof}
A geometric property of $\bP^n$:

\begin{lemma}\label{constant-functions-Pn-minus-points}
  Let $n>1$ and $p\neq q$ be points of $\bP^n$, then all functions
  \begin{center}
  \begin{enumerate}[(i)]
  \item $\bP^n\setminus\{p\}\to \Z$
  \item $\bP^n\setminus\{p,q\}\to \Z$
  \item $\bP^n\setminus\{p\}\to R$
  \item $\bP^n\setminus\{p,q\}\to R$
  \end{enumerate}
  \end{center}
  are constant.
\end{lemma}

\begin{proof}
  We start with (iv).
  Let $f:\bP^n\setminus\{p,q\}\to R$.
  For the charts $U_0=\{[x_0:\dots:x_n]\mid x_0\neq 0\}$ and $U_1=\{[x_0:\dots:x_n]\mid x_0\neq 0\}$, we can assume $p\in U_0, p\notin U_1$ and $q\in U_1, q\notin U_0$.
  Then $f_{|U_0\setminus\{p\}}$ can be extended to $U_0$ by \Cref{ext} and an analogous extension exisits on $U_1$.
  These extensions glue with $f$ to a function $\widetilde{f}:\bP^n\to R$ which agrees with $f$ on $\bP^n\setminus\{p,q\}$.
  By \Cref{const}, $\widetilde{f}$ is constant and therefore $f$ is constant.
  This carries over to functions $\bP^n\setminus\{p,q\}\to \mathrm{Bool}$ since $\mathrm{Bool}\subseteq R$ and thus also to any $\bP^n\setminus\{p,q\}\to \Z$,
  which shows (ii).
  (i) and (iii) follow from (ii) and (iv).
\end{proof}
We proceed by extendending \Cref{Matthias} to subspaces of $\bP^n$ which can be constructed like $\bP^1$:

\begin{lemma}\label{line-bundle-on-line}
  Let $M\subseteq R^{n+1}$ be a submodule with $\|M=R^2\|$.
  Then $\Gr(1,M)\subseteq \bP^n$ and the map
  \begin{align*}
    \Z\times \KR &\to (\Gr(1,M)\to \KR) &\\
    (d,l_0) &\mapsto (L\mapsto L^{\otimes d}) &\text{ for $d\geq 0$} \\
    (d,l_0) &\mapsto (L\mapsto \Hom_{\Mod{R}}(L^{\otimes d},R)) &\text{ for $d< 0$} 
  \end{align*}
  is an equivalence.
\end{lemma}

\begin{proof}
  We prove a proposition, so we have an $R$-linear isomorphism $\phi:R^2\to M$ and for each $d:\Z$, we get a commutative triangle:
  \begin{center}
  \begin{tikzcd}
    \Gr(1,M)\ar[rr,"\OO(d)"] && \KR \\
    & \Gr(1,R^2)\ar[lu,"\phi"]\ar[ur,swap,"\OO(d)"] &
  \end{tikzcd}
\end{center}
by restricting $\phi$ to each line in $\Gr(1,R^2)$.
This shows that the map from \Cref{Matthias} and from the statement are equal as maps to $(V:\Mod{R})\times \|V=R^2\| \times (V\to\KR))$,
which proves the claim.
\end{proof}

\begin{theorem}
  The map
  \begin{align*}
    \Z\times \KR &\to (\bP^n\to \KR) \\
    (d,l_0) &\mapsto (x\mapsto l_0\otimes \OO(d)(x))
  \end{align*}
  is an equivalence.
\end{theorem}

\begin{proof}
  It is enough to show that the map is surjective, by the same reasoning as in the proof of \Cref{Matthias}.
  Let $L:\bP^n\to \KR$. First we determine the degree of $L$.
  Let $p\neq q$ be points in $\bP^n$ and $M\subseteq R^{n+1}$ be the span of $p$ and $q$ as submodules of $R^{n+1}$.
  Then $\|M=R^2\|$ and we can use the inverse $i$ of the map in \Cref{line-bundle-on-line} to define $d\colonequiv \pi_1(i(L_{|\Gr(1,M)}))$.
  The integer $d$ is independent of the choice of $p$ and $q$:
  If we let $p$ vary, we get a function of type $\bP^n\setminus\{q\}\to \Z$ which is constant by \Cref{constant-functions-Pn-minus-points}.
  The same applies for $q$
  and the two subsets $\bP^n\setminus\{p\}$ and $\bP^n\setminus\{q\}$ cover $\bP^n$.

  In the following we consider only $L$ such that $d$ as constructed above is $0$.
  This means that on each line $\Gr(1,M)$, $L$ will be constant.
  So for $p,x:\bP^n$,
  and $x\neq p$ we can construct an equality $P_x:L(x)=L(p)$  by restricting $L$ to $\Gr(1,\langle x,p\rangle)$ and applying \Cref{line-bundle-on-line}.
  So we have $P:(x:\bP^n\setminus\{p\})\to L(x)=L(p)$ and for $q\neq p$ we can construct
  $Q:(y:\bP^n\setminus\{q\})\to L(y)=L(q)$ analogously.
  
  The claim follows if we show that $L$ is constant on all of $\bP^n$.
  Since, overall, we show the proposition that the map from the statement merely has a preimage,
  we can assume $a:L(p)=R^1$ and $b:L(q)=R^1$ and get:
  \[ \left((x:\bP^n\setminus\{p,q\})\mapsto a^{-1}P_x^{-1}Q_xb\right) : \bP^n\setminus\{p,q\} \to R^\times\]
  which is constantly $\lambda$ by \Cref{constant-functions-Pn-minus-points}.
  So $P$ and $Q$ can be corrected using $\lambda,a$ and $b$ to yield a global proof of constancy of $L$.
\end{proof}

\newpage

\section*{Appendix 1: Horrock's Theorem}

We present a constructive proof of the the following special case of the {\em affine}
Horrocks Theorem \cite{Lam}, V.2, for a commutative ring $A$. We
essentially follow Nashier-Nichols' Proof of Horrocks Theorem, as presented in \cite{Lam}, IV.5.
(Another constructive proof can be found in \cite{lombardi-quitte}, XVI.4).

\begin{lemma}\label{Horrocks}
  If an ideal of $A[X]$ divides a principal ideal $(f)$ with $f$ monic then it is itself a principal ideal.
\end{lemma}

Let $L$ and $M$ be such that $L\cdot M = (f)$. We can then write $f = \Sigma u_iv_i$ with $u_i$ in $L$ and
$v_i$ in $M$. Using $f$ monic, we then have $L = (u_1,\dots,u_n)$ and $M = (v_1,\dots,v_n)$.
The strategy of the proof is to build comaximal monoids $S_1,\dots,S_l$ in $A$ \cite{lombardi-quitte},
XIV.1, such that $L$ is generated by a monic polynomial in each $A_{S_j}[X]$.

\subsection{Formal computation of gcd}


 If we have a list $u_1,\dots,u_n$ of polynomials over a field we can compute the gcd of this list
$(g) = (u_1,\dots,u_n)$ and $g$ is either $0$ (in the case where all the polynomials $u_1,\dots,u_n$ are $0$)
 or a monic polynomial. More generally, over a local ring $A$ which is residually discrete of maximal
 ideal $m$, we can compute a polynomial $g$ in $A[X]$ such that 
 $(g) = (u_1,\dots,u_n)$ modulo $m$
 and $g$ is either $0$ (in the case where all the polynomials $u_1,\dots,u_n$ are $0$ modulo $m$)
 or a monic polynomial.

In general, if we are over an arbitrary ring $A$, we can interpret this computation formally as
follows (\cite{lombardi-quitte}, XIV.1). We build a binary tree of root $A$, where at each node,
we intuitively force an element of $A$ to be in the Jacobson radical
or to be invertible modulo this ideal.
To each node is associated a pair of finite sets $I;U$
of elements in $A$ and we associate the monoid $S(I;U) = M(U) + (I)$ where $M(U)$ is the multiplicative
monoid generated by $U$ and $(I)$ the ideal generated by $I$.

Corresponding to the formal computation of the gcd, we get a binary tree where we have at each leaf
a ring $A_{S(I;U)}$ and a polynomial $g$ in $A_{S(I;U)}[X]$, which is monic or $0$, and
such that $(g) = (u_1,\dots,u_n)$ modulo the ideal generated by $I$ in $A_{S(I;U)}$.


If we do this for each branch, we get a list of monoids $S_1,\dots,S_l$
that are {\em comaximal} (\cite{lombardi-quitte}, XIV.1): if $s_i$ in $S_i$ then $1 = (s_1,\dots,s_l)$.

\subsection{Application to Horrocks' Theorem}

We assume that $(u_1,\dots,u_n)$ divides $(f)$, with $f$ monic.
The goal is to build comaximal monoids $S_1,\dots,S_l$ with $(u_1,\dots,u_n)$ principal
and generated by a monic polynomial in $A_{S_j}[X]$.

Note that $(u_1,\dots,u_n)$ contains the monic polynomial $f$.

We first build a binary tree which corresponds to the formal computation of the gcd of
$u_1,\dots,u_n$ as described above. For each branch $I;U$ we have a monic polynomial
$g$ in $A_{S(I;U)}[X]$, which belongs to $(u_1,\dots,u_n)$, and
such that $(u_1,\dots,u_n) = (g)$ modulo\footnote{Since $(u_1,\dots,u_n)$
contains a monic polynomial, the case where $g=0$ can only happen if $0$ belongs to in $S(I;U)$ and in this
case, we can replace $g$ by $1$.} the ideal generated by $I$.

The next Lemma is Lemma IV.5.1, \cite{Lam} in Lam's presentation of
Nashier-Nichols' Proof of Horrocks Theorem.

\begin{lemma}
  Let $R$ be a ring with an ideal $J$ contained in the Jacobson radical
  and $L$ an ideal of $R[X]$ which contains a monic polynomial. We consider
  the reduction modulo $J$
  $$\pi: R[X]\mapsto (R/J)[X]$$
  Any monic polynomial of $\pi(L)$ can be lifted to a monic polynomial in $L$.
\end{lemma}

 Using this Lemma, we get a monic polynomial $h$ in $(u_1,\dots,u_n)$ in $A_{S(I;U)}[X]$
 and such that $h$ generates $(u_1,\dots,u_n)$ modulo $(I)$.
 We can now use that $I$ is contained in the Jacobson radical of $A_{S(I;U)}$ and the
 following second Lemma, which corresponds to Proposition IV.5.2 of \cite{Lam},
 to conclude that we actually have $(h) = (u_1,\dots,u_n)$ in $A_{S(I;U)}[X]$.

\begin{lemma}
  Let $R$ a ring, $J$ an ideal of $R$ contained in the Jacobson radical of $R$. If
  we have $L\cdot M = (f)$ with $f$ monic in $R[X]$, and $L$ contains a monic polynomial
  $h$ such that $L = (h)$ in $(R/J)[X]$ then $L = (h)$ in $R[X]$.
\end{lemma}

\begin{proof}
  Since $L$ contains $L\cap J$ and $L\cdot M = (f)$ with $f$ regular (since $f$ is monic),
  we can find $K$   such that $L\cdot K = L\cap J$. Indeed, we have $(L\cap J)\cdot M \subseteq (f)$
  and so we can write $f K = (L\cap J)\cdot M$. We then have $f (L\cdot K) = (L\cap J)\cdot L\cdot M = f (L\cap J)$
  and so $L\cdot K = L\cap J$ since $f$ is monic.
  We then have $L\cdot K = 0$ modulo $J$ and hence $K = 0$ modulo $J$ since $L$ contains $f$
  which is monic.
  This means $L\cap J = L\cdot J$. Then we have $L = (h) + L\cdot J$.
  The result then follows from the fact that $h$ is monic and from Nakayama's Lemma, as in Lam \cite{Lam}:
  the module $P = L/(h)$ is a finitely generated module over $R$, since $f$ is monic, and satisfies
  $P\subseteq JP$ and $J$ is contained in the Jacobson radical of $R$, so $P = 0$ by Nakayama's Lemma.
\end{proof}

\begin{corollary}
  We can find comaximal elements $s_1,\dots,s_l$ such that $(u_1,\dots,u_n)$ is principal and generated by a
  monic polynomial in each $A_{s_j}[X]$. Since these monic polynomials are uniquely determined
  we can patch these generators and get that $(u_1,\dots,u_n)$ is principal in $A[X]$\footnote{If $A$ is not
  connected, the generator of $(u_1,\dots,u_n)$ may not be monic: if $e(1-e)=1$ then the ideal $(eX+(1-e))$
  divides the ideal $(X)$.}.
\end{corollary}

\newpage

\section*{Appendix 2: proof of Proposition \ref{units}}

The goal is to prove the following result.

\begin{proposition}\label{units}
  Let $A$ be a {\em connected} ring and let $t_{ij}$ be a family of  invertible elements in $T_{ij}(A)$ such that $t_{ii} = 1$
  and $t_{ij}t_{jk} = t_{ik}$. There exists an integer $N$ and $s_i$ invertible in $T_i$ such that $t_{ij} = (X_j/X_i)^Ns_j/s_i$ 
\end{proposition}

\begin{proof}
  
Using \Cref{stand}, \Cref{nilpunit} and \Cref{connected}, we can assume without loss of generality, that
$t_{ij} = (X_i/X_j)^{N_{ij}} u_{ij}$, for some $N_{ij}$ in $\Z$, where $u_{ij}(p)$ is invertible for $p = 0$
and all other coefficients $u_{ij}(p)$ for $p\neq 0$
are nilpotent. By looking at the relation  $t_{ik} = t_{ij}t_{jk}$ when we quotient by nilpotent elements, we see that
$N_{ij} = N$ does not depend on $i,j$.

  Each $u_{ij}$ is such that $u_{ij}(p)$ unit for $p=0$ and
  all $u_{ij}(p)$ nilpotent for $p\neq 0$. The goal is to find $s_i$ unit in $T_i(A)$ such that
  $u_{ij} = s_j/s_i$.

  Let $M$ be the monoid of $p = (p_0,\dots,p_n)$ such that $p_0+\dots+p_n = 0$.
  Let $M_i$ be the submonoid of $p$ in $M$ such that $p_i<0$ and $p_l\geqslant 0$ if $l\neq i$
  and $M_{ij}$ the monoid of $p$ in $M$ such that $p_i<0$ and $p_j<0$ and $p_l\geqslant 0$ if $l\neq i,j$.
  We let $L_{i}$ be the free $A$-module generated by $X^p$ for $p$ in $M_i$ and $L_{ij}$ be the
  free $A$-module generated by $X^p$ for $p$ in $M_{ij}$. We can write $T_{i}(A)$ as the direct sum $A+L_i$
  and $T_{ij}(A)$ as the direct sum $A + L_i + L_j + L_{ij}$. Furthermore we have $L_iL_j\subseteq A + L_i + L_j + L_{ij}$
  and $L_iL_{ij}\subseteq L_i + L_{ij}$.

  The element $u_{01}$ can be uniquely decomposed as $u_{01} =  a + x + y + z$ with $a$ unit in $A$ and
  $x$ in $L_0$ and $y$ in $L_1$ and $z$ in $L_{01}$. We claim that we can assume $x = y = 0$ by multiplying
  $u_{01}$ by units in $T_0(A)$ and in $T_1(A)$.
  The argument is similar to the one of \Cref{nilpotent}. We let $J$ be the nilpotent ideal generated by all
  $u_{01}(p)$ for $p\neq 0$. We have $x$ in $JL_0$ and $y$ in $JL_1$ and $z$ in $JL_{01}$.
  We multiply $u_{01}$ by $1 - a^{-1}x$, getting
  $$(1-a^{-1}x) u_{01} = a' + x' + y' + z'$$
  with $a'$ unit in $A$ and $x'$ in $J^2L_0$ and $y'$ in $JL_1$ and $z'$ in $JL_{01}$.
  If we multiply by $1- a'^{-1}x'$ we get
  $$(1-a'^{-1}x') (a'+x'+y'+z') = a'' + x'' + y'' + z''$$
  with $a''$ unit in $A$ and $x''$ in $J^3L_0$ and $y''$ in $JL_1$ and $z''$ in $JL_{01}$.
  Keeping doing this, we see that there is a unit $u$ in $T_0(A)$ such that
  $uu_{01}$ is in $A + JL_1+ JL_{01}$. By a similar reasoning, and using that $L_1L_{01}\subseteq L_1 + L_{01}$, we find a unit
  $v$ in $T_1(A)$ such that $vuu_{01}$ is in $A + JL_{01}$.
  
  I claim now that if $u_{01}$ is in $A + L_{01}$ then $u_{01}$ is actually in $A$.

  For this, we use the relation $u_{01}= u_{0l}u_{l1}$, that is
  $$u_{01}(p) = u_{0l}(p)u_{l1}(0) + u_{0l}(0)u_{l1}(p) + \Sigma_{q+r = p, q\neq 0, r\neq 0}u_{0l}(q)u_{l1}(r)\leqno{(1)}$$

  \begin{lemma}
    For all $l\neq 0,1$, we have 
    $u_{0l}(p) = u_{1l}(p) = 0$ if $p_l>0$.
  \end{lemma}
  
  \begin{proof}
  We let $K$ be the ideal generated by coefficients $u_{0l}(p)$ 
  and $u_{l1}(p)$ for $p_l>0$ and $I$ the (nilpotent) ideal generated by all coefficients $u_{0l}(p)$ and
  $u_{l1}(p)$ for $p\neq 0$. 
  We show $K\subseteq KI$.

  $(1)$ can be written as
  $$u_{0l}(p)u_{l1}(0) = u_{01}(p) - u_{0l}(0)u_{l1}(p) - \Sigma_{q+r = p, q\neq 0, r\neq 0}u_{0l}(q)u_{l1}(r)\leqno{(2)}$$
  (recall that $u_{l1}(0)$ is a unit). We use $(2)$ to show that $u_{0l}(p)$ is in $KI$ if $p_l>0$.

  The element $u_{0l}$ in $T_{0l}(A) = A+L_0+L_l+L_{0l}$, so we have $u_{0l}(p) = 0$ if $p$ is not in $M_0$ and $p_l>0$.
  So we can assume $p$ in $M_0$.

  We have $u_{01}(p) = 0$ if $p$ is in $M_0$, since $u_{01}$ is in $A+L_{01}$. We also have $u_{l1}(p) = 0$ if $p$ in $M_0$ since
  $u_{l1}$ is in $A+L_1+L_l+L_{1l}$.
  So, if $p$ is in $M_0$, we can rewrite $(2)$ as
  $$u_{0l}(p)u_{l1}(0) = - \Sigma_{q+r = p, q\neq 0, r\neq 0}u_{0l}(q)u_{l1}(r)$$
  Each member in the sum $u_{0l}(q)u_{l1}(r)$ is in $KI$ since $q_l+r_l = p_l>0$ and hence $q_l>0$ or $r_l>0$.
  So, we have $u_{0l}(p)$ in $KI$ if $p_l>0$.

  Similarly, using
  $$u_{0l}(0)u_{l1}(p) = u_{01}(p) - u_{0l}(p)u_{l1}(0) - \Sigma_{q+r = p, q\neq 0, r\neq 0}u_{0l}(q)u_{l1}(r)$$
  we show that $u_{l1}(p)$ is in $KI$ if $p_l>0$. So we have $K\subseteq KI$, as desired.

  We then deduce $K\subseteq KI^n$ for all $n$ and hence $K=0$ since $I$ is nilpotent.
  \end{proof}

  Let $p$ be in $M_{01}$. We can find $l\neq 0,1$ such that $p_l>0$. If $p = q+r$ we have $q_l>0$ or $r_l>0$ and hence
  by the Lemma, we have $u_{0l}(q) = 0$ or $u_{l1}(r) = 0$. Hence the equality $(1)$ shows that we have $u_{01}(p)= 0$.
  Since $u_{01}$ is in $A+L_{01}$, it follows that $u_{01}$ is actually a unit in $A$.

  W.l.o.g. we can assume $u_{01}= 1$. We then have $u_{0l} = u_{01}u_{1l} = u_{1l}$ in $T_{0l}(A)\cap T_{1l}(A) = T_l(A)$
  and we take $s_l = u_{0l} = u_{1l}$.
\end{proof}

Note that the proof is much simpler if $A$ is both connected {\em and} reduced (for example if $A$ is a discrete field).

Using Horrocks Theorem and the technique of Quillen patching, both explained in a constructive setting in \cite{lombardi-quitte}, we
obtain from this Proposition a constructive proof of the following purely algebraic result.

\begin{theorem}
  Let $A$ be a ring which is connected and such that $\Pic(A) = 0$. If $M_i$ is a projective module of rank $1$ over
  $T_i(A)$ and we have isomorphisms $t_{ij}:M_i\otimes T_{ij}(A) \equiv M_j\otimes T_{ij}(A)$ with $t_{ik} = t_{jk}\circ t_{ij}$
  and $t_{ii} = \mathsf{id}$, then there is an integer $d$ and invertible elements $s_i$ in $T_i(A)$
  and isomorphisms $\psi_i:M_i\simeq T_i(A)$ such that $\psi_j(t_{ij} x) = \psi_i(x)(X_j/X_i)^d s_j/s_i $ for all
  $x$ in $M_i\otimes T_{ij}(A)$.
\end{theorem}

\newpage

\section*{Appendix 3: Quillen Patching}

We reproduce the argument in Quillen's paper \cite{Quillen}, as simplified in \cite{lombardi-quitte}.
This technique of Quillen Patching has been replaced by the equivalence in Proposition \ref{Matthias}.

Let $P(X)$ be a presentation matrix of a finitely presented module $M$ over a ring $A[X]$. We say that $M$
is extended from $A$ if the matrix $P(X)$ and the matrix $P(0)$ presents isomorphic modules \cite{lombardi-quitte}, Ch. XVI.

 Quillen Patching can be presented as the following result.

\begin{theorem}\label{QP}
  If we have $M$ finitely presented of $A[X]$ and $f_1,\dots,f_n$ comaximal elements of $A$
  such that each $M\otimes_{A[X]} A[1/f_i][X]$ is extended from $A[1/f_i]$ then $M$ is extended from $A$.
\end{theorem}

Let us reformulate this result in the setting of Synthetic Algebraic Geometry \cite{draft}.
If $A$ is a finitely presented $R$-algebra, we know \cite{draft}, Theorem 7.2.3, that there is an
equivalence between $\fpMod{A[X]}$ and ${\fpMod{R}}^{\Spec(A)\times R}$, which to a module $M$
associates the family $M\otimes_{A[X]} (x,r)$ for $x:\Spec(A)$ and $r:R$. Conversely, to a family $L~x~r$
of finitely presented $R$-modules, we associate the finitely presented $A[X]$-module $\prod_{x:\Spec(A)}\prod_{r:R}L~x~r$.
We deduce from this the following characterisation of extended modules from $A$: the module corresponding to the
family $L~x~r$ is extended from $A$ if, and only if, we have $\prod_{x:\Spec(A)}\prod_{r:R}L~x~r = L~x~0$.

We can then reformulate Theorem \ref{QP} as follows.

\begin{corollary}
  If we have $f_1,\dots,f_n$ comaximal elements of $A$ and $\prod_{x:D(f_i)}\prod_{r:R}L~x~r = L~x~0$ for $i=1,\dots, n$
  then $\prod_{x:\Spec(A)}\prod_{x:R}L~x~r = L~x~0$.
\end{corollary}

It follows then from local choice that we have.

\begin{corollary}\label{QP1}
  If we have $\prod_{x:\Spec(A)}\norm{\prod_{r:R}L~x~r = L~x~0}$ then we have $\norm{\prod_{x:\Spec(A)}\prod_{x:R}L~x~r = L~x~0}$.
\end{corollary}

We can now compare Quillen's argument which uses Corollary \ref{Pic1} and Theorem \ref{QP}
instead of its refinement Corollary \ref{Matthias1}.

\begin{proposition}\label{trivial1}
  For all $V:\bP^n\rightarrow \KR$ we have $\norm{\prod_{s:R^n}V([1:s]) = V([0:1:0:\cdots :0])}$.
\end{proposition}

\begin{proof}
  We define $L:R^{n-1}\rightarrow (\bP^1 \to \KR)$ by $L~t~[x_0:x_1] = V([x_0:x_1:x_0t])$.
  Let $s=(s_1,\dots,s_{n}):R^{n}$. We apply Corollary \ref{Pic1} and we get, for all $s_2,\dots,s_n$
  \[
   \norm{V([1:s]) = L~(s_2,\dots,s_n)~[1:s_1] = L~(s_2,\dots,s_n)~[0:1]}
   \rlap{.}
   \]
   Using Corollary \ref{QP1}, we get $\norm{\prod_{s:R^n}V([1:s]) = V([0:1:0:\cdots :0])}$.
\end{proof}

If $P$ and $Q$ are two idempotent matrix of the same size, let us write $P\simeq Q$ for expressing that $P$ and $Q$ presents
the same projective module (which means that there are similar, which is in this case is the same as being equivalent).

If we have a projective module on $A[X]$, presented by a matrix $P(X)$, this module is extended
precisely when we have $P(X)\simeq P(0)$.

\begin{lemma}
  If $S$ is a multiplicative monoid of $A$ and $P(X)\simeq P(0)$ on $A_S[X]$ then there exists
  $s$ in $S$ such that $P(X+sY)\simeq P(X)$ in $A[X]$.
\end{lemma}

\begin{lemma}
  The set of $s$ in $A$ such that $P(X+sY)\simeq P(X)$ is an ideal of $A$.
\end{lemma}

\begin{corollary}
  If we have $M$ projective module of $A[X]$ and $S_1,\dots,S_n$ comaximal multiplicative monoids of $A$
  such that each $M\otimes_{A[X]} A_{S_i}[X]$ is extended from $A_{S_i}$ then $M$ is extended from $A$.
\end{corollary}

Let us reformulate in synthetic term this result. Let $A$ be a f.p. $R$-algebra and $L:\Spec(A)\rightarrow B\Gm^{\A^1}$.
Then $L$ corresponds to a projective module of rank $1$ on $A[X]$. We can form
$$T(x) = \prod_{r:R}L~x~r = L~x~0$$
and $\|T(x)\|$ expresses that $L~x$ defines a trivial line bundle on $\A^1 = \Spec(R[X])$.
It is extended exactly when we have
$\|{\prod_{x:\Spec(A)}T(x)}\|$. We can then use Zariski local choice to state.

\begin{proposition}\label{c2}
  We have the implication $(\prod_{x:\Spec(A)}\|T(x)\|)\rightarrow \|\prod_{x:\Spec(A)}T(x)\|$.
\end{proposition}

\newpage

\section*{Appendix 4: Classical argument}

We reproduce a message of Brian Conrad in MathOverflow \cite{conrad-mathoverflow-16324}.

\medskip

``We know that the Picard group of projective $(n-1)$-space over a field $k$ is $\Z$
generated by $\OO(1)$.
This underlies the proof that the automorphism group of such a projective space is $\PGL_n(k)$.
But what is the automorphism group of $\bP^{n-1}(A)$ for a general ring $A$? Is it $\PGL_n(A)$?
It's a really important fact that the answer is yes.
But how to prove it? It's a shame that this isn't done in Hartshorne.

By an elementary localization, we may assume $A$ is local.
In this case we claim that $\Pic(\bP^{n-1}(A))$ is infinite cyclic generated by $\OO(1)$.
Since this line bundle has the known $A$-module of global sections,
it would give the desired result if true by the same argument as in the field case.
And since we know the Picard group over the residue field, we can twist
to get to the case when the line bundle is trivial on the special fiber. How to do it?

\medskip

 Step 0: The case when $A$ is a field. Done.

 \medskip

 Step 1: The case when $A$ is Artin local.
 This goes via induction on the length, the case of length $0$ being Step $0$
 and the induction resting on cohomological results for projective space over the residue field.

  \medskip

 Step 2: The case when $A$ is complete local noetherian ring. This goes
 using Step 1 and the theorem on formal functions (formal schemes in disguise).

  \medskip

 Step 3: The case when $A$ is local noetherian.
 This is faithfully flat descent from Step 2 applied over $A~\widehat{}$

 \medskip
 
 Step 4: The case when $A$ is local:
 descent from the noetherian local case in Step 3 via direct limit arguments.

\medskip
 
QED''

\printindex

\printbibliography

\end{document}